\documentclass[12pt]{article}
\usepackage{amsmath}
\usepackage{amsthm}
\usepackage{amssymb}
\usepackage{graphicx}
\usepackage{booktabs}
\usepackage{threeparttable}
\title{The scaling limits for Wiener sausages in random environments}
\author{Chien-Hao Huang\footnote{email address: chienhaohuang@ntu.edu.tw} \\
Department of Mathematics \\ 
National Taiwan University}

\newtheorem{theorem}{\bf Theorem}[section]
\newtheorem{lemma}{\bf Lemma}[section]

\numberwithin{equation}{section}

\date{}

\begin{document}
\maketitle

{\bf Keywords.} {Random walk, random environment, Wiener sausage, continuum limit, Wiener chaos}

\begin{abstract}
   We consider the statistical mechanics of a random polymer with random walks and disorders in $\mathbb{Z}^d$. The walk collects random disorders along the way and gets nothing if it visits the same site twice. In the continuum and weak disorder regime, the partition function as a random variable converges weakly to a Wiener Chaos expansion when the dimension is lower than the critical dimension, which is four. A finite temperature case in one dimension is also discussed. The last case suggests that the end-point behavior of the polymer is $t^{2/3}$.
\end{abstract}

\section{Introduction}

\subsection{The model}

We consider the model that a particle walks on $\mathbb{Z}^d$ lattice, and $\{\omega_x\}_{x\in \mathbb{Z}^d}$ is an i.i.d. random field with mean zero, variance one and finite exponential moment under the measure $\mathbb{P}$. The movement of the particle is decribed by a simple symmetric random walk $S_n$ with the measure $P$. $\mathbb{P}$ and $P$ are independent. The particle utilizes the field when it visits a new site $x$, then the random field no longer exists. Let $\mathcal{R}_n$ denote the set of sites visit by the particle in the first $n$ steps, that is, 
$$\mathcal{R}_n :=\{\; x\in \mathbb{Z}^d\; |\; S_i=x,1\leq i\leq n \},$$
and $R_n=|\mathcal{R}_n|$ is the cardinality of $\mathcal{R}_n$.

The Hamiltonian is defined as follows
\begin{equation}
H_n :=   \sum_{x\in \mathbb{Z}^d} \; (\beta\omega_x+h)  \cdot 1_{x\in \mathcal{R}_n}.
\end{equation}
$\beta\geq 0$ is related to the inverse temperature and $h\in\mathbb{R}$ represents an external force.
The polymer measure is
\begin{equation}
P^{\omega}_{n}(S) := \frac{1}{Z_n} \exp ( H_n) \cdot P(S)
\end{equation}
where
\begin{equation}
Z_n := E(\exp ( H_n))
\end{equation}
is called the quenched partition function. We also define the annealed partition function
\begin{equation}
\mathbb{E} Z_n := E(e^{(\lambda(\beta)+h) R_n} )
\end{equation}
where $\lambda(\beta) =\log \mathbb{E} e^{\beta \omega_0}$. In the literature, the annealed model is called Wiener sausage. \cite{DV79} is the seminal paper when $h<-\lambda(\beta)$, and \cite{vdBT91} considered $h>-\lambda(\beta)$ with the random walks replaced by Brownian motions.

In this paper, we dicuss the case $h=-\lambda(\beta)$ and the intermediate regime $\beta=\beta_n$, that is,

\begin{equation}
Z_n(\beta_n) :=  E\left(e^{ \sum_{x\in \mathbb{Z}^d} \; (\beta_n\omega_x-\lambda(\beta_n))  \cdot 1_{x\in \mathcal{R}_n}}\right) , \;\;\mathbb{E}Z_n(\beta_n) =1.
\end{equation}

This intermediate regime was discussed in \cite{CSZ17}. Define $\eta_x(\beta) =\beta^{-1}(e^{\beta \omega_x -\lambda(\beta)}-1)$ and recall that \begin{equation}
1_{x\in \mathcal{R}_n}=1_{T_x\leq n}.
\end{equation} 
\cite{CSZ17} consider the following expansion of $Z_n(\beta_n)$,
\[\begin{array}{cl}
\displaystyle  Z_n(\beta_n) &\displaystyle = E \left[ \prod_{x\in Z^d}\limits  \left(1+\left(e^ {\beta_n \omega_x -\lambda(\beta_n)}-1\right)1_{T_x\leq n}\right) \right] = E \left[ \prod_{x\in Z^d}\limits  \left(1+\beta_n\eta_x(\beta_n)1_{T_x\leq n}\right) \right] \\
& \displaystyle =  E \left[ 1+ \beta_n\sum_{x\in \mathbb{Z}^d}\limits \eta_x(\beta_n ) 1_{T_x\leq n} \right. \\
&  \;\;\;\;\;\;\; \displaystyle + \frac{1}{2!}\beta_n^2\sum_{x,y\in \mathbb{Z}^d;\; x\neq y}\limits  \eta_x(\beta_n ) 1_{T_x\leq n} \; \eta_y(\beta_n ) 1_{T_y\leq n}\\
&   \;\;\;\;\;\;\; \displaystyle \left. +\frac{1}{3!}\beta_n^3\sum_{x,y,z\in \mathbb{Z}^d; \; x\neq y \neq z}\limits  \eta_x(\beta_n ) 1_{T_x\leq n}  \;\eta_y(\beta_n ) 1_{T_y\leq n} \;\eta_z(\beta_n ) 1_{T_z\leq n} + \cdots\right]
\end{array}\]
\[\left.\begin{array}{cl}
&\displaystyle =  1+ \beta_n\sum_{x\in \mathbb{Z}^d}\limits \eta_x(\beta_n ) \;E(1_{T_x\leq n}) \\
&   \;\;\;\; \displaystyle + \frac{1}{2!}\beta_n^2\sum_{x,y\in \mathbb{Z}^d; x\neq y}\limits  \eta_x(\beta_n ) \eta_y(\beta_n ) \; E(1_{T_x\leq n}  1_{T_y\leq n})\\
&    \;\;\;\; \displaystyle +\frac{1}{3!}\beta_n^3\sum_{x,y,z\in \mathbb{Z}^d; x\neq y \neq z}\limits  \eta_x(\beta_n ) \eta_y(\beta_n )\eta_z(\beta_n ) \; E(1_{T_x\leq n} 1_{T_y\leq n} 1_{T_z\leq n}) + \cdots .
\end{array}\right\}(*)\]
Notice that $\eta_x$'s are i.i.d. with $\mathbb{E}\eta_x = 0$ and $\sigma^2(\eta_x) \approx 1$ when $\beta \approx 0$.

This expansion is going to converge to a non-trivial limit $Z^{W}$ as $\beta_n$ approaches $0$.
We need two ingredients to show the convergence. First, the scaling limits for $k$-point function $E(1_{T_{\sqrt{n}x_1} \leq n}\cdots 1_{T_{\sqrt{n}x_k} \leq n})$. Second, $\beta_n$ is chosen such that the variance of the first-order term in $(*)$ is finite in the limit.
\begin{equation}\label{betan}
Var(\beta_n\sum_{x\in Z^d} \eta_x E1_{x\in \mathcal{R}_n} )\sim \beta_n^2 \sum_{x\in Z^d} [E1_{x\in \mathcal{R}_n}]^2=  \beta_n^2 \: E\times E' (S^1[1,n] \cap S^2 [1,n]) .
\end{equation}
$E\times E' (S^1[1,n] \cap S^2 [1,n]) $ is called the intersection of independent ranges \cite{Chen09}. Define $$J_n^{(p)} := \# \{S^1[1,n]\cap S^2[1,n]\cap ...\cap S^p[1,n]\}.$$ 
Let $p=2$ and $J_n=J_n^{(2)}$,
%page 156 Chen09
\[\begin{array}{cll}
 d=1, & \;\;\;  \displaystyle \frac{1}{\sqrt{n}} J_n  &\;  \rightarrow_{law}\; \displaystyle \min_{i=1,2}\max_{0\leq s \leq 1} B^i(s) - \max_{i=1,2}\min_{0\leq s \leq 1} B^i(s) ,\\
 d=2, & \;\;\; \displaystyle \frac{(\log n)^2}{n} J_n &\;  \rightarrow_{law} \; \displaystyle 2\pi^2 \;\alpha\left([0,1]^2\right),\\
 d=3, & \;\;\;  \displaystyle \frac{1}{\sqrt{n}} J_n &\;  \rightarrow_{law} \; \displaystyle 2 \gamma^2 \;\alpha\left([0,1]^2\right),\\
 d=4, & \;\;\;  \displaystyle \frac{1}{\log n} J_n &\;  \rightarrow_{law} \; 2 (2\pi)^{-2}\; \gamma^2 \;N(0,1)^2,  
\end{array}\]
and $d\geq 5, J_n <\infty$ almost surely. $\gamma=\gamma_d$ is the escape rate of $d$-dimensional simple random walk, and $\alpha(\cdot)$ is the 2-multiple mutual-intersection local time of independent Brownian motions $B^1(t)$ and $B^2(t)$ \cite{Chen09}. Because of \eqref{betan}, $\beta_n$ is chosen as follows
\[\begin{array}{cc}
d=1, & \;\;\;  \displaystyle \beta_n=\frac{\hat{\beta}}{n^{1/4}} ,\\
d=2, & \;\;\; \displaystyle \beta_n=\frac{\log n }{\sqrt{n}} \hat{\beta},\\
d=3, & \;\;\;  \displaystyle \beta_n=\frac{\hat{\beta}}{n^{1/4}} .
\end{array}\]

\subsection{Main results}%1.2

Denote $\bar{p}_n(x):= \left(\frac{d}{2\pi n}\right)^{d/2} \exp(-\frac{d|x|^2}{2n})$.
For $d=2$, $g_{t}(x) := \pi\bar{p}_t(x) $, for $d=3$, $g_{t}(x) :=\gamma \bar{p}_t(x)$.

\begin{theorem}
Let $[0,t]^k_{<} :=\{(t_1,t_2,...,t_k); 0\leq t_1\leq t_2\leq ...\leq t_k\leq t\}$ and $\Sigma_k$ be the permutation group over $\{1,...,k\}$. Also set $t_0=0, x_{\sigma(0)}=0$. The expansion $(*)$ of $Z_{Nt}(\beta_{N})$ converges weakly to
a Wiener chaos expansion
\begin{equation}
Z_t^{W} := 1+ \sum_{k=1}^\infty \frac{\hat{\beta}^k}{k!}\int\cdots\int_{(R^d)^k} \psi_t(x_1,...,x_k)W(dx_1)\cdots W(dx_k).
\end{equation}
as $N\to \infty$.
For d=2,3,
\begin{equation}
\psi_{t}(x_1,\dots,x_k):= \sum_{\sigma\in\Sigma_k} \int_{[0,t]^k_{<}} \prod_{m=1}^k  g_{t_m-t_{m-1}}(x_{\sigma(m)}-x_{\sigma(m-1)}) \; dt_1 \cdots dt_k .
\end{equation}

For d=1, 
\begin{equation}
\psi_{t}(x_1,\dots,x_k):= \sum_{\sigma\in\Sigma_k} P( \min_{0\leq s \leq t} B_s \leq x_{\sigma(1)}<\cdots< x_{\sigma(k)}\leq \max_{0\leq s \leq t} B_s). 
\end{equation}
\end{theorem}

\subsection{Discussions}

%E[\mathbb{E}(\sum_{x\in \mathbb{Z}^d}  \omega_x\cdot 1_{T_x\leq n} )(\sum_{x\in \mathbb{Z}^d}  \omega_x\cdot 1_{T'_x\leq n})]=\sum_{x\in Z^d}[E 1_{x\in \mathcal{R}_n}]^2 

The overlap is defined as 
\begin{equation}
\sum_{x\in Z^d}E\times E'(1_{x\in \mathcal{R}_n} 1_{x\in \mathcal{R}'_n})
=E\times E' (S^1[1,n] \cap S^2 [1,n])=J_n. \end{equation}

From the behavior of the overlap $J_n$, when $d=1,2,3$, the model is called disorder relevant; $d\geq 5$ the model is disorder irrelevant.

\if 0
First order,
\begin{equation}
\beta_N\sum_{x\in Z^2} \eta_x E1_{x\in R_n} \rightarrow_{law}  \int_{\mathbb{R}^2} \left(\int_{0}^{1}\frac{1}{ t} \exp(-\frac{|x|^2}{t}) dt \right)  \hat{\beta} W(dx)
\end{equation}
\fi

$d=4$ is believed to be the critical dimension. We are not going to discuss this case here. See more details about the critical dimension in \cite{CSZ17b}.

There is a similar model called directed polymers in random environments (DPRE) \cite{Com17}. The Hamiltonian is 
\begin{equation}
H_n :=     \sum_{i=1}^n  (\beta\omega(i,S_i) -\lambda(\beta))
\end{equation}
The random potential $\omega$ is defined on space and time. Unlike the model we discussed in this paper, DPRE is directed. For the one-dimesional case, the weak disorder limit is the solution of stochastic heat equation \cite{AKQ14}, however, we don't see the analogous result here. On another hand, the critical dimension for DPRE is two, and here we have four instead.

The last comment we would like to make is that, for $d=1$, $h=0$ and finite $\beta$, the the scale of the end-point position of the walk seems to be $n^{2/3}$. Details are in Section 3.

\section{Proof of Theorem 1.1}

\subsection{Preliminary results}

Denote
$p_n(x) :=P(S_n=x)$, $G_n(x) := \sum_{m=1}^n p_m(x)$,
$G_n:=G_n(0)$, $\gamma_n := G_n^{-1}$.\\

\iffalse
\begin{equation}
	P(T_x \leq 2n ) =\sum_{k=1}^n P(S_{2k} =x) \gamma_{2n-2k}\sim \sum_{k=1}^n \left( \frac{1}{2\pi k} \exp(-\frac{|x|^2}{2k})\right) \frac{\pi}{\log (2n-2k)}
\end{equation}

$$E1_{\sqrt{N} x_1\in \mathcal{R}_{Nt_1}}=P(T_{\sqrt{2n}x} \leq 2n ) 
\sim \sum_{k=1}^n \left( \frac{1}{2 k} \exp(-\frac{2n|x|^2}{2k})\right) \frac{1}{\log (2n-2k)}$$
$$\sim \left( \int_{0+}^{1-}   \frac{1}{ 2t} \exp(-\frac{|x|^2}{t}) dt \right) \frac{1}{\log (2n)}
$$
\fi

We recall Lemma 5.1.3 in \cite{Chen09}.\\

\noindent
\textbf{Lemma A.} {\it For any $x\in\mathbb{Z}^d$,
\begin{equation}\label{5.1.8}
P(S_k=x) =\sum_{j=1}^{k} P(T_x=j)P(S_{k-j}=0).
\end{equation}
Consequently,
\begin{equation}\label{5.1.9}
\sum_{k=1}^n P(S_k=x) =\sum_{k=1}^{n} P(T_x=k)G_{n-k}.
\end{equation}
}

We first consider the one-point function. Notice that 
\begin{equation}
1_{T_x\leq n} =\sum_{i=1}^{n} 1_{T_x=i}.
\end{equation}
\begin{lemma}\label{lem:1point} %L2.1 1-point
$x\neq 0$,
\begin{equation}
	P(T_{x} \leq n ) \geq G_n^{-1}\sum_{k=1}^{n} P(S_{k} = x) 
\end{equation}
and
\begin{equation}
	G_{\epsilon n} P(T_{x} \leq n ) \leq \sum_{k=1}^{(1+\epsilon )n}  P(S_k =x ).
\end{equation}
\end{lemma}

\begin{proof} They are typical bounds. First, since $G_n$ is increasing in $n$,
%Lower bound 
$$	G_n P(T_{x} \leq n ) \geq \sum_{k=1}^{n} P(T_{ x}=k)G_{n-k}=
\sum_{k=1}^{n} P(S_{k} = x)  $$
and
%Upper bound
$$ G_{\epsilon n} P(T_{x} \leq n ) \leq \sum_{k=1}^{n}  P(T_{x} =k ) G_{(1+\epsilon) n -k}$$
$$\leq  \sum_{k=1}^{(1+\epsilon )n}  P(T_{x} =k ) G_{(1+\epsilon) n -k}
=  \sum_{k=1}^{(1+\epsilon )n}  P(S_k =x ).
$$

Equalities are from \eqref{5.1.9}.\end{proof}

We now discuss the $k$-point function. Recall that $\Sigma_k$ is the permutation group over $\{1,...,k\}$.

\begin{lemma}\label{lem:kpoint} %L2.2 k-point
$d=2,3$ and $0, x_1,...,x_k$ are distinct points,
	\begin{equation*}
	P(T_{ x_1} \leq n, \; T_{ x_2} \leq n, ... , \; T_{ x_k} \leq n )\leq G_{\epsilon n}^{-k} \;\cdot 
	\end{equation*}
	\begin{equation}\label{kupper}
\sum_{\sigma\in\Sigma_k} \sum_{1\leq j_1 < \cdots < j_k \leq (1+k\epsilon)n } P(S_{j_1}=x_{\sigma(1)})P( S_{j_2-j_1}=x_{\sigma(2)}-x_{\sigma(1)})\cdots P(  S_{j_{k}-j_{k-1}}=x_{\sigma(k)}-x_{\sigma(k-1)} )
	\end{equation}
For the lower bound,
	\begin{equation*}
	P(T_{ x_1} \leq n, \; T_{ x_2} \leq n)\geq 
	\end{equation*}
	\begin{equation}\label{klower2}
    G_n^{-2} (\sum_{1\leq j <k \leq n} P(S_j=x_1,\;  S_{ k} =x_2)+\sum_{1\leq k <j \leq n} P(S_j=x_1,\;  S_{ k} =x_2))+\mbox{higher order terms}.
	\end{equation}
Moreover, for general $k$,
	\begin{equation*}
P(T_{ x_1} \leq n, \; T_{ x_2} \leq n,...,\; T_{ x_k} \leq n  )\geq 
\end{equation*}
\begin{equation*}
G_n^{-k} \sum_{\sigma\in\Sigma_k} \sum_{1\leq j_1 < \cdots < j_k \leq n } P(S_{j_1}=x_{\sigma(1)})P( S_{j_2-j_1}=x_{\sigma(2)}-x_{\sigma(1)})\cdots P(  S_{j_{k}-j_{k-1}}=x_{\sigma(k)}-x_{\sigma(k-1)} )
\end{equation*}
\begin{equation}\label{klower} +\mbox{higher order terms}.\end{equation}
\end{lemma}
Remark. The ``higher order terms'' means that after we take $n\rightarrow\infty$, it at least has one more factor $\frac{1}{\sqrt{n}^d}$ than the first term.\\
\begin{proof}
First we exam the upper bound \eqref{kupper}.
For $k=2$, exactly from (5.3.22) in \cite{Chen09} p. 153,

$$P(T_{ x_1} \leq n, \; T_{ x_2} \leq n)$$
$$=  \sum_{\sigma\in\Sigma_2} \sum_{1\leq j_1 < j_2 \leq n } P(T_{ x_{\sigma(1)}} =j_1, \; T_{ x_{\sigma(2)}} =j_2 )$$
$$\leq G_{\epsilon n}^{-2} \;\cdot 
\sum_{\sigma\in\Sigma_2} \sum_{1\leq j_1 <j_2 \leq (1+2\epsilon)n } P(S_{j_1}=x_{\sigma(1)})P( S_{j_2-j_1}=x_{\sigma(2)}-x_{\sigma(1)})$$
%$$P(T_{ x_1} \leq n, \; T_{ x_2} \leq n, ... , \; T_{ x_k} \leq n )$$
%$$=  \sum_{\sigma\in\Sigma_k} \sum_{1\leq j_1 < \cdots < j_k \leq n } P(T_{ x_{\sigma(1)}} =j_1, \; T_{ x_{\sigma(2)}} =j_2, ... , \; T_{ x_{\sigma(k)}} =j_k )$$
\noindent
For general $k$, do the induction on the $T_{x_k}$.\\

For the lower bound \eqref{klower}, again the case $k=2$,
$$P(T_{ x} \leq n, \; T_{ y} \leq n)=\sum_{j=1}^{n} P(T_{ x} =j, \; T_{ y} \leq n)$$
$$=\sum_{j=1}^{n} \sum_{k=1}^{n} P(T_{ x} =j, \; S_{ k} =y, \; S_{k+1}\neq y, S_{k+2}\neq y,..., S_{n}\neq y )   $$
$$=\sum_{1\leq j <k \leq n} P(T_{ x} =j, \; S_{ k} =y, \; S_{k+1}\neq y, S_{k+2}\neq y,..., S_{n}\neq y )   $$
$$\;\;\;\;\;\;\;\; +\sum_{1\leq k <j \leq n} P(T_{ x} =j, \; S_{ k} =y, \; S_{k+1}\neq y, S_{k+2}\neq y,..., S_{n}\neq y )   $$
$$=I+II. $$
We have
$$I = \sum_{1\leq j <k \leq n} P(T_{ x} =j) P( S_{ k-j} =y-x) P(S_{1}\neq 0, S_{2}\neq 0,..., S_{n-k}\neq 0 )$$
$$\geq  \sum_{1\leq j <k \leq n} P(T_{ x} =j) P( S_{ k-j} =y-x) \gamma_n$$
$$\geq \gamma_n G_n^{-1} \sum_{1\leq j <k \leq n} P(S_j=x,\;  S_{ k} =y)$$
and
$$II= \sum_{1\leq k <j \leq n} P(T_{ x} \geq k, \; S_{ k} =y, \; S_{k+1}\neq y, S_{k+2}\neq y,..., S_{n}\neq y, \; S_{k+1}\neq x,..., S_{j-1}\neq x, S_j=x  )$$
$$=\sum_{1\leq k <j \leq n} P(T_{ x} \geq k, \; S_{ k} =y)P(T_0 >n-k, \; T_{x-y}=j-k  )$$
$$=\sum_{1\leq k <j \leq n} P(T_{ x} \geq k, \; S_{ k} =y)P(T_0 >j-k, \; T_{x-y}=j-k  )  P(T_{y-x} >n-j)$$
$$\geq \sum_{1\leq k <j \leq n} P(T_{ x} \geq k, \; S_{ k} =y)P(T_0 >j-k, \; T_{x-y}=j-k  ) \cdot P(T_{y-x} >n). $$

Apply Lemma \ref{lem:kpointmore} below to the middle term of the last line $P(T_0 >j-k, T_{x-y}=j-k  )$, then
$$II \geq \sum_{1\leq k <j \leq n} P(T_{ x} \geq k, \; S_{ k} =y) G_n^{-1} P( T_{x-y}=j-k  ) \cdot P(T_{y-x} >n).$$
Continue with steps in \cite{Chen09} p. 154,
\[\begin{array}{ccl}
\displaystyle II &\displaystyle \geq & \displaystyle G_n^{-2} \sum_{1\leq k <j \leq n} P(S_{ k} =y, \; S_{ j} =x)\cdot P(T_{y-x} >n) \\
   &     &  \displaystyle - G_n^{-1} P(T_x \leq n) P(T_{x-y}\leq n) \sum_{k=1}^n P(S_k =y-x) \cdot P(T_{y-x} >n).
\end{array}\]
Notice that $P(T_{\sqrt{n} (y-x)} >n) =  1- P(T_{\sqrt{n} (y-x)} \leq n) \rightarrow 1$ when $d=2,3$ by Lemma \ref{lem:1pointsacle} below. The last term has an extra $P(\cdot)$, which gives a factor $\frac{1}{\sqrt{n}^d}$ more than the previous one. This finishes the case $k=2$.
We use symbols $T$ and $S$ for $T_x$ and $S_j$ respectively to explain the procedure again. Let's say the proof above for the 2-point case \eqref{klower2} is a process from $TT$ to $SS$. $TT$ first splits into $TS$ and $ST$ right before the line $I+II$. $TS$ in $I$ and $ST$ in $II$ give $SS$. For the 3-point function, we fix the first hitting time $j_1$ of $x_{1}$ and work on the last two time spots $j_2$ and $j_3$. The 3-point case $T[TT]$ first becomes $T[SS]$ from the 2-point case. Then we consider the first two time spots and make $[TS]S$ become $[SS]S$ from $I$.

For general $k$, $T\cdots TT\rightarrow T\cdots TSS$, then continues by induction. Thus, the proof is complete.
\end{proof}

\begin{lemma} \label{lem:kpointmore} %L2.3 k-point more detail
	$$\sum_{k=1} ^n P(T_x =k) \leq G_n \sum_{j=1} ^n  P( T_x =j , T_0 >j).$$
\end{lemma}

\begin{proof} From \cite{Spi76} p. 112, $P(S_k=x, T_0 >k) =  P(T_x =k)$. So we need to prove
$$\sum_{k=1} ^n P(S_k=x, T_0 >k) \leq G_n \sum_{j=1} ^n  P( T_x =j , T_0 >j).$$
First,
$$P(S_k=x, T_0 >k) = P(S_k=x, T_x\leq k , T_0 >k) =\sum_{j=1}^{k} P(S_k=x, T_x =j , T_0 >k) $$
$$=\sum_{j=1}^{k} P( T_x =j , T_0 >j) P(S_{k-j}=0, T_{-x} > k-j).$$
Again,
$$\sum_{k=1} ^n P(S_k=x, T_0 >k)  = \sum_{j=1} ^n \sum_{k=j} ^n  P( T_x =j , T_0 >j) P(S_{k-j}=0, T_{-x} > k-j)$$
$$= \sum_{j=1} ^n  P( T_x =j , T_0 >j) \sum_{i=0} ^{n-j}   P(S_{i}=0, T_{-x} > i) \leq G_n \sum_{j=1} ^n  P( T_x =j , T_0 >j).$$
\end{proof}

\subsection{Capacity for $d=2, 3$}

Recall that $p_n(x) :=P(S_n=x)$, $\bar{p}_n(x):= \left(\frac{d}{2\pi n}\right)^{d/2} \exp(-\frac{d|x|^2}{2n})$.
For $d=2$, $g_{t}(x) := \pi\bar{p}_t(x) $, for $d=3$, $g_{t}(x) :=\gamma \bar{p}_t(x)$.
Let $k_N=\log N$, if $d=2$; $k_N=\sqrt{N}$ if $d=3$.
%appl of L2.1
\begin{lemma} \label{lem:1pointsacle}
	d=2,3, $x_1\neq 0$,
	\begin{equation}
	k_N P(T_{\sqrt{N} x_1} \leq Nt_1)\rightarrow \int_0^{t_1} g_{s}(x_1) ds.
	\end{equation}
	
\end{lemma}

\begin{proof}
%From Revesz pp.214
It is well known that $d=2$, $G_n \sim \frac{1}{\pi} \log n$ and 
$d=3$, $G_n\rightarrow 1/\gamma$.\\
With Lemma \ref{lem:1point}, we only need to prove 
$$\sqrt{N}^{d-2}\sum_{k=1}^{Nt_1} P(S_{k} = \sqrt{N}x_1) \sim  \int_0^{t_1} \bar{p}_{s}(x_1) ds . $$
Also note that
$p_k(\sqrt{N}x_1) = \bar{p}_k(\sqrt{N}x_1) +O(\frac{1}{k^{d/2+1}})$. For $n<< N$, $\sum_{k=1}^n p_k(\sqrt{N}x_1) \approx 0$, and for $n\approx N$, $\sum_{k=n+1}^N O(\frac{1}{k^{d/2+1}}) = O(1/N^{d/2})$. The rest is the Riemann sum for $\sum \bar{p}_k(\sqrt{N}x_1)$.
\end{proof}

For the k-point function.
Let $[0,t]^k_{<} =\{(t_1,t_2,...,t_k); 0\leq t_1\leq t_2\leq ...\leq t_k\leq t\}$.

\begin{lemma}\label{lem:kpointsacle} %L2.5
$d=2,3$ and $0, x_1,...,x_k$ are distinct points,
	\begin{equation*}
	(k_N)^k P(T_{\sqrt{N} x_1} \leq Nt, \; T_{\sqrt{N} x_2} \leq Nt, ... , \; T_{\sqrt{N} x_k} \leq Nt )
	\end{equation*}
	\begin{equation}
	\sim \sum_{\sigma\in\Sigma_k} \int_{[0,t]^k_{<}} \prod_{m=1}^k  g_{t_m-t_{m-1}}(x_{\sigma(m)}-x_{\sigma(m-1)}) dt_1 \cdots dt_k .
	\end{equation}
\end{lemma}
\begin{proof}
Use Lemma \ref{lem:kpoint}.	
\end{proof}

Here is a byproduct, the case for the point-to-point function.

\begin{lemma}\label{lem:condition} %2.6
$d=2,3$ and	$0, x, x_1$ are distinct points, 
	\begin{equation}
	k_N \cdot P(T_{\sqrt{N} x_1} \leq Nt |\; S_{Nt}=\sqrt{N}x)\sim \int_0^{t}g_{t_1}(x_1) g_{t-t_1}(x-x_1) dt_1 / g_t(x) .
	\end{equation}
\end{lemma}

\begin{proof}
$$ P(T_{\sqrt{N} x_1} =j_1 ,\; S_{Nt}=\sqrt{N}x)$$
$$= P(S_{Nt}=\sqrt{N}x |\; T_{\sqrt{N} x_1} =j_1) P(T_{\sqrt{N} x_1} =j_1 )=P(S_{Nt}=\sqrt{N}x |\; S_{j_1}=\sqrt{N}x_1) P(T_{\sqrt{N} x_1} =j_1 ) . $$
\end{proof}

\subsection{Convergence for $d=2,3$}
Let $\beta_N =\hat{\beta} a_N$, $a_N= (\sqrt{N}/\log N)^{-1}$ if $d=2$; $a_N= N^{-1/4}$ if $d=3$.
Moreover,
\begin{equation}
\psi_{N,t}(x_1,\dots,x_k):= a_N^k P(T_{\sqrt{N} x_1} \leq Nt, \dots, T_{\sqrt{N} x_k} \leq Nt )
\end{equation}
and
\begin{equation}
\psi_{t}(x_1,\dots,x_k):= \sum_{\sigma\in\Sigma_k} \int_{[0,t]^k_{<}} \prod_{m=1}^k  g_{t_m-t_{m-1}}(x_{\sigma(m)}-x_{\sigma(m-1)}) \; dt_1 \cdots dt_k .
\end{equation}

The space scaling $v_N :=1/(\sqrt{N})^d$, then we have $a_N/\sqrt{v_N}=k_N$.
From the previous section,
\begin{equation}
\lim_{N\rightarrow \infty}  v_N^{-k/2} \psi_{N,t}(x_1,\dots,x_k) =\psi_{t}(x_1,\dots,x_k) .
\end{equation}
Recall that
$$\eta_{N,x}:=\frac{1}{a_N} \left( e^{\beta_N \omega_{\sqrt{N} x} -\Lambda(\beta_N)}-1 \right) . $$

Let an index set $T_{N} =\{N^{-1/2}k: k\in \mathbb{Z}^d\}\subset \mathbb{R}^d$ and a family of polynomial chaos expansions $(\Psi_{N}(\eta_{N,x}))$. Let
$\mu_0(x)=0$, $\sigma_0(x)=\hat{\beta}$.
We are ready to
check the conditions for the case $\mu_0=0$ in Theorem 2.3 \cite{CSZ17} which we state here:\\

(i) $\eta_{N,x}$'s are uniformly integrable, $v_N\to 0$ as $N\to \infty$ and $\lim_{N\to \infty} Var(\eta_{N,x}) =\sigma_0^2$.\\

(ii) There exists $\psi_t$ with $\psi_t\in L^2((\mathbb{R}^d)^k)$ for every $k\in \mathbb{N}$ such that
\begin{equation}
\lim_{N\to \infty} || v_N^{-k/2} \psi_{N,t} -\psi_t||_{L^2((\mathbb{R}^d)^k)}=0.
\end{equation}

(iii) 
\begin{equation}
\lim_{\ell\to\infty} \limsup_{N\to \infty} \sum_{I\subset T_N, |I|>\ell} (\sigma_N^2)^{|I|} \psi_{N,t}(I)^2=0.
\end{equation}

\noindent
{\bf Proof of Theorem 1.1}\\

(i) $\eta_{N,x}$'s are uniform bounded.
$$\mathbb{E}(\eta_{N,x})^4 \leq \frac{1}{a_N^4} O(\beta_N^4)=O(1)$$
And $\lim_{N\to \infty} Var(\eta_{N,x})=\hat{\beta}^2$.\\

ii) 
\if 0
$$N p_{N(t-s)}(\sqrt{N}(y-x)) \rightarrow \bar{p}_{t-s} (y-x)$$
$$ \sum_{\sigma\in\Sigma_k} \sum_{1\leq j_1 < \cdots < j_k \leq Nt } P(S_{j_1}=x_{\sigma(1)})P( S_{j_2-j_1}=x_{\sigma(2)}-x_{\sigma(1)})\cdots P(  S_{j_{k}-j_{k-1}}=x_{\sigma(k)}-x_{\sigma(k-1)} )$$
$$\leq$$
\fi
Since $P(S_{Nt}=\sqrt{N} x)\leq C_N\bar{p}_t(x)$ for all x and $C_N \rightarrow 1$ as $N\to \infty$,
\[\begin{array}{cl}
      & \displaystyle \int|v_N^{-k/2} \psi_{N,t} (\cdot)|^2 \;\prod_{i=1}^k dx_i \\
 \leq & \displaystyle C_N\int \;\prod_{i=1}^k dx_i \left[\sum_{\sigma\in\Sigma_k} \int_{[0,t]^k_{<}}\bar{p}_{t_{j}-t_{j-1}}(x_{\sigma(j)}-x_{\sigma(j-1)})  \prod_{j=1}^k dt_j\right]^2 \\
    :=& \displaystyle C_N I_{k,t}.
\end{array}\]

Notice that
\[\begin{array}{cl}
     & \displaystyle \int_{[0,t]^k_{<}} \bar{p}_{t_{j}-t_{j-1}}(x_{\sigma(j)}-x_{\sigma(j-1)})\prod_{j=1}^k dt_j \\
\leq & \displaystyle \prod_{j=1}^k \int_0^t p_s(x_{\sigma(j)}-x_{\sigma(j-1)}) ds.
\end{array}\]

So
\[\begin{array}{cll}
\displaystyle I_{k,t} & \displaystyle \leq &  \displaystyle k! \sum_{\sigma\in\Sigma_k} \int \;\prod_{i=1}^k dx_i \left[\prod_{j=1}^k \int_0^t p_{s_1}(x_{\sigma(j)}-x_{\sigma(j-1)}) ds_1\right]^2 \\
&=   & \displaystyle k! \sum_{\sigma\in\Sigma_k} \int_0^t \int_0^t ds_1 ds_2
  \left[\int \prod_{i=1}^k dx_i \;\prod_{j=1}^k p_{s_1}(x_{\sigma(j)}-x_{\sigma(j-1)}) p_{s_2}(x_{\sigma(j)}-x_{\sigma(j-1)})\right]\\
&=   & \displaystyle k! \sum_{\sigma\in\Sigma_k} \int_0^t \int_0^t ds_1 ds_2
\left[\int \prod_{i=1}^k dx_i \; p_{s_1}(x_j-x_{j-1}) p_{s_2}(x_j-x_{j-1})\right]\\
&=   & \displaystyle (k!)^2 \int_0^t \int_0^t ds_1 ds_2 \left[ \int p_{s_1} (y_i) p_{s_2}(y_i) dy_i \right]^k\\
&=   & \displaystyle C (k!)^2 \int_0^t \int_0^t ds_1 ds_2 \frac{1}{ (s_1+s_2)^{d/2}}\\
&\leq& \displaystyle C (k!)^2 \int_0^t \int_0^t ds_1 ds_2  \frac{1}{s_1^{d/4} s_2^{d/4}}
\end{array}\]
which is finite since $d=2,3$.\\

iii) First, $Var(\eta_{N,x})=\sigma_N^2 \leq B$ for some $B\in (0,\infty)$. And
\[\begin{array}{cl}
     & \displaystyle \sum_{I\subset T_N, |I|>\ell} (\sigma_N^2)^{|I|} \psi_{N,t}(I)^2 \\
\leq & \displaystyle \sum_{k>\ell}B^k \frac{1}{k!} \sum_{(z_1,...,z_k)\in (T_N)^k} \psi_{N,t}(z)^2  \\
=    & \displaystyle \sum_{k>\ell} B^k\frac{1}{k!} ||v_N^{-k/2} \psi_{N,t}||^2_{L^2((\mathbb{R}^d)^k)}. \\
\end{array}\]

\if 0
$$\sum_{k=1} \frac{\hat{\beta}^{2k}}{(k!)^2}\int\cdots\int_{(R^2)^k} dx_1\cdots dx_k \left[\sum_{\sigma\in\Sigma_k} \int_{[0,t]^k_{<}} \prod_{m=1}^k  g_{t_m-t_{m-1}}(x_{\sigma(m)}-x_{\sigma(m-1)}) dt_1 \cdots dt_k\right]^2$$
\fi

From \cite{Chen09} Theorem 2.2.3 p. 29 and p. 41 with $p=2$, and for each $k$, the $L^2$ norm has an upper bound
$(k!)^{\frac{d-2}{2}} \; t^{\frac{4-d}{2}k} C^k$. Thus, 
$$\leq \sum_{k>\ell} B^k\frac{1}{k!} \; (k!)^{\frac{d-2}{2}} \; t^{\frac{4-d}{2}k} C^k . $$
The latter term goes to $0$ when $\ell \to \infty$ for both $d=2,3$. The completes the case $d=2,3$ in Thoerem 1.1.\qed\\

Remark. For the point to point case,
$$(k_N)^k P(T_{\sqrt{N} x_1} \leq Nt, \dots, T_{\sqrt{N} x_k} \leq Nt |S_{Nt}=\sqrt{N}x)$$
$$\sim \frac{1}{g_t(x)}\sum_{\sigma\in\Sigma_k} \int_{[0,t]^k_{<}} \prod_{m=1}^k  g_{t_m-t_{m-1}}(x_{\sigma(m)}-x_{\sigma(m-1)}) g_{t-t_{k}}(x-x_{\sigma(k)})dt_1 \cdots dt_k. $$
And
\begin{equation}
\psi^c_{N,(t,x)}(x_1,\dots,x_k):= a_N^k P(T_{\sqrt{N} x_1} \leq Nt, \dots, T_{\sqrt{N} x_k} \leq Nt |S_{Nt}=\sqrt{N}x),
\end{equation}

\begin{equation}
\psi^c_{t,x}(x_1,\dots,x_k):= \frac{1}{g_t(x)}\sum_{\sigma\in\Sigma_k} \int_{[0,t]^k_{<}} \prod_{m=1}^k  g_{t_m-t_{m-1}}(x_{\sigma(m)}-x_{\sigma(m-1)}) g_{t-t_{k}}(x-x_{\sigma(k)})dt_1 \cdots dt_k .
\end{equation}
The superscript $c$ stands for ``constrained''.
We then have

\begin{equation}
\lim_{N\rightarrow \infty}  v_N^{-k/2} \psi^c_{N,(t,x)}(x_1,\dots,x_k) =\psi^c_{t,x}(x_1,\dots,x_k).
\end{equation}

In total, say $t=1$,
$$ Z_N(x) := E(\exp (\beta_N  H_N)|S_N=\sqrt{N}x)\rightarrow_{law} Z^{W,c}(x). $$

One can define the point-to-point partition function, 
$$Z^W(x):= Z^{W,c}(x)\bar{p}_1(x).$$

\subsection{$d=1$}

For $d=1$,
take $\beta_{N}=\hat{\beta}N^{-1/4}$. We first have
$$P(T_{\sqrt{N} x} \leq Nt)\rightarrow P(T_x \leq t).$$

For the k-point function
$$P(T_{\sqrt{N} x_1} \leq Nt, \; T_{\sqrt{N} x_2} \leq Nt, ..., \; T_{\sqrt{N} x_k} \leq Nt )\rightarrow P(T_{x_1} \leq t, \;T_{x_2} \leq t, ..., \; T_{ x_k} \leq t), $$
and the last term is equal to $$P_0( \min_{0\leq s \leq t} B_s \leq x_1<x_2<...< x_k\leq \max_{0\leq s \leq t} B_s).$$
So it is easy to see the case $d=1$ in Theorem 1.1.

\section{A special case for $d=1$ when $\beta$ is finite}

In this section, we consider the case $h=0$. In the mean time, we choose the underlying process as one-dimensional Brownian motion $\{B(t)\}_{t\geq 0}$, and the random environment is modeled by a two-sided Brownian motion $\{W(x)\}_{-\infty<x<\infty}$. $B$ and $W$ are independent. The Hamiltonian we have is

\begin{equation}
H_t :=   \int 1_{x\in \mathcal{R}_t} W(dx)  =\int^{M_t}_{m_t} W(dx) =W_{M_t} -W_{m_t},
\end{equation}
and
\begin{equation}
Z_t := E(\exp ( \beta H_t)) 
\end{equation}
where $M_t;=\max_{0\leq s \leq t} B_s$ and $m_t:= \min_{0\leq s \leq t} B_s$.
The annealed patition function is easily computed.

\begin{equation}
\mathbb{E}Z_t := E(\exp ( \frac{1}{2}\beta^2 R_t)) ,
\end{equation}
where $R_t=M_t-m_t$ is the range of the Brownian motion $B$ up to time $t$. By rescaling $W$, 

$$Z_t =^d E\exp ( \beta t^{1/3} (W_{M_{t}/t^{2/3}} -W_{m_{t}/t^{2/3}})  ).$$
Denote $T=t^{1/3}$,
$$Z_t=^d E\exp ( \beta T (W_{M_{T^3}/T^{2}} -W_{m_{T^3}/T^{2}})  )$$
$$=E\exp ( \beta T (W_{M_{T}/T} -W_{m_{T}/T})  ). $$
The last equality is obtained by rescaling $B$. \cite{Fe51} calculated the explicit joint density of $(m_T,M_T)$. If the brownian motion $B$ is ballistic, namely, of order $T$, the price to pay is $\exp(-cT)$, and the energy term gives $\exp(\tilde{c}T)$ as well. So
$$\frac{1}{T}\log E\exp ( \beta T (W_{M_{T}/T} -W_{m_{T}/T})  ) \approx O(1).$$
Then we have
\begin{equation}\label{free}
t^{-1/3} \log Z_t \approx O(1)
\end{equation}
from above discussions. From the scaling relation $\frac{1}{3}=2\chi -1$,
\eqref{free} suggests that the scale $\chi$ of the end-point of the polymer in the 1D case is $t^{2/3}$, which is the same as $(1+1)$-directed polymer.

\iffalse
By Laplace's method
\begin{equation}
\lim_{T\rightarrow \infty}\frac{1}{T} \log E\exp ( \beta T (W_{M_{T}/T} -W_{m_{T}/T})  ) =\sup_{r>0}  \left( \sup_{r>M>0}\beta (W_M-W_{M-r})-\frac{r^2}{2}\right)\mathbb{P}-a.s.
\end{equation}
$\frac{r^2}{2}$ is the rate funtion of $R_t$.
We then have
\begin{theorem}
	$$ t^{-1/3} \log Z_t \rightarrow_{law}  
\sup_{r>0}  \left( \sup_{r>M>0} \beta(W_M-W_{M-r})-\frac{r^2}{2}\right) \;\;\mbox{as}\;\; t\rightarrow \infty.$$
\end{theorem}
\fi

\section*{Acknowledgement}

The research of Chien-Hao Huang was partially supported by Ministry of Science and Technology grants MOST 106-2115-M-002-010.

\end{document}